\documentclass[11pt]{article}
\usepackage{amsmath,amsthm,amssymb}
\usepackage[margin=1.3in]{geometry}

\title{Sums of dilates%, I
}
\author{Boris Bukh}
\date{}
\newcommand*{\Z}{\mathbb{Z}}
\newcommand*{\N}{\mathbb{N}}
\newcommand*{\R}{\mathbb{R}}
\newtheorem{theorem}{Theorem}
\newtheorem{lemma}[theorem]{Lemma}
\newtheorem{corollary}[theorem]{Corollary}
\newtheorem{definition}[theorem]{Definition}
\newtheorem{question}[theorem]{Question}

\newtheorem{claim}{Claim}
\newcommand*{\abs}[1]{\lvert #1\rvert}
\newcommand*{\norm}[1]{\lVert #1\rVert}
\newcommand*{\blambda}{\bar{\lambda}}

\begin{document}
\maketitle
\begin{abstract}
The $\lambda$-dilate of a set $A$ is $\lambda\cdot A=\{\lambda a:a\in A\}$.
We give an asymptotically sharp lower bound on the size of sumsets of the form
$\lambda_1\cdot A+\dotsb+\lambda_k\cdot A$ for arbitrary integers
$\lambda_1,\dotsc,\lambda_k$ and integer sets $A$. We also
establish an upper bound for such sums, which is similar to, but often 
stronger than Pl\"unnecke's inequality. 
\end{abstract}

\section*{Introduction}
For sets $A,B$ in an abelian group $G$ their sumset is $A+B=\{a+b:a\in A,\ b\in B\}$.
For $\lambda\in \Z$ the dilate of $A$ by $\lambda$ is
$\lambda\cdot A=\{\lambda a : a\in A\}$. Expressions of the form
\begin{equation}\label{genericdilatesum}
\lambda_1\cdot A+\dotsb+\lambda_k\cdot A
\end{equation}
appear frequently in combinatorial number theory. For $k=2$ they appeared in
the work of Nathanson, O'Bryant, Orosz, Ruzsa, and Silva on binary linear forms\cite{cite:binary_forms}. For small $k$ they appeared in the proofs of
sum-product estimates in $\Z/p\Z$ of Garaev and Katz-Shen\cite{cite:garaev_sumprod,cite:katz_shen_sumprod}. 
They played an important part in the solution
to a problem of Ruzsa on symmetric linear equations \cite{cite:ruzsatriv}.

The problem of giving a lower bound on a sum of the form 
\eqref{genericdilatesum} first occurred in the work of {\L}aba and Konyagin
on distances in well-distributed planar sets 
\cite{cite:konyagin_laba_transcendental}.
They treated the case of $A+\lambda\cdot A$ for $G=\R$ and 
transcendental~$\lambda$.  The general problem of giving a lower bound 
on the sum of dilates
when $G=\Z$ was treated by Nathanson\cite{cite:nathanson_forms} who in particular proved that 
$\abs{A+2\cdot A}\geq 3\abs{A}-2$ and $\abs{A+\lambda\cdot A}\geq 7\abs{A}/2-O(1)$ for
positive $\lambda\neq 1,2$. Our first result is a sharp
lower bound on the size of $A+3\cdot A$:
\begin{theorem}\label{thmAthreeA}
For every finite set $A\subset \Z$ we have
$\abs{A+3\cdot A}\geq 4\abs{A}-O(1)$.
\end{theorem}
It is interesting that there are two essentially different examples
that achieve the lower bound in the theorem above. Besides
the arithmetic progression $A=\{1,\dotsc,n\}$
the lower bound is achieved by the set $A=\{1,2,4,5,\dotsc,3m+1,3m+2\}$.
The proof of theorem~\ref{thmAthreeA} is an easy, but computationally 
involved, induction argument. However, for the reason
that is explained after the proof of theorem~\ref{thmAthreeA},
any similarly-structured induction argument has to get 
computationally even more involved for $A+\lambda\cdot A$ with 
$\lambda=4$ or greater.

The main result of this paper is an almost sharp lower bound on any sum of
dilates of the form~\eqref{genericdilatesum}. Instead of sharp $O(1)$
error term of theorem~\ref{thmAthreeA}, it has the weaker $o(\abs{A})$ 
error term. 
\begin{theorem}\label{mainthm}
For every vector $\blambda=(\lambda_1,\dotsc,\lambda_k)\in \Z^k$ of $k$ coprime integers we have
\begin{equation*}
\abs{\lambda_1\cdot A+\dotsb+\lambda_k\cdot A}\geq (\abs{\lambda_1}+\dotsb+\abs{\lambda_k})\abs{A}-o(|A|)
\end{equation*}
for every finite set $A\subset \Z$ with the error term $o(|A|)$ depending on $\blambda$ only.
\end{theorem}
The case when $\lambda_1,\dotsc,\lambda_k$ are not coprime can be
reduced to the case when they are coprime via the relation
\begin{equation*}
\lambda_1\cdot A+\dotsb+\lambda_k\cdot A=\gcd(\lambda_1,\dotsc,\lambda_k)\cdot
\left(\frac{\lambda_1}{\gcd(\lambda_1,\dotsc,\lambda_k)}\cdot A+\dotsb+\frac{\lambda_k}{\gcd(\lambda_1,\dotsc,\lambda_k)}\cdot A\right).
\end{equation*}
Theorem~\ref{mainthm} is sharp apart from the $o(\abs{A})$ term as
witnessed by $A=\{1,\dotsc,n\}$. 

The problem of bounding~\eqref{genericdilatesum} can be seen as a special
case of the problem of establishing inequalities between two or more
sums of the form \eqref{genericdilatesum}. For instance, if $\abs{A+A}$
is small, how small does $A+\lambda\cdot A$ need to be? Since
$A+\lambda\cdot A\subset \underbrace{A+\dotsb+A}_{\lambda+1\text{ times}}$
for positive integer $\lambda$, the classical Pl\"unnecke inequality
\cite[Corollary~5.2]{cite:ruzsa_plunnecke} tells us that 
$\abs{A+A}\leq K\abs{A}$ implies $\abs{A+\lambda\cdot A}\leq K^{\lambda+1}\abs{A}$.
This estimate is far from being sharp.
\begin{theorem}\label{gentriagineq}
If either $\abs{A+A}\leq K\abs{A}$ or $\abs{A-A}\leq K\abs{A}$, then $\abs{\lambda_1\cdot A+\dotsb+\lambda_k\cdot A}\leq K^p\abs{A}$ where 
\begin{equation*}
p=7+12\sum_{i=1}^k \log_2(1+\abs{\lambda_i}).
\end{equation*}
\end{theorem}
The logarithmic dependence on $\lambda_i$ is optimal, as seen by
considering $A+\lambda\cdot A$ with $A=\{1,\dotsc,n\}$.
The constants $7$ and $12$ are certainly not optimal, and the dependence
on $k$ is probably not optimal.

Theorem~\ref{gentriagineq} allows one to strengthen the main result of
\cite{cite:ruzsatriv} to
\begin{theorem}
For a symmetric linear equation $\lambda_1 x_1+\dotsb+\lambda_k x_k=
\lambda_1 y_1+\dotsb+\lambda_k y_k$ let $R(N)$ be the size of the 
largest $A\subset \{1,\dotsc,N\}$ not containing a solution to the equation
in distinct integers. Then if $k\geq 3$ and $\lambda_i\neq 0$ for all $1\leq i\leq k$,
then
\begin{equation*}
R(N)=O\bigl(N^{\frac{1}{2}-\frac{1}{c(k)\log\norm{\blambda}_1}}\bigr).
\end{equation*}
\end{theorem}
The proof can be obtained from the proof of 
Theorem~1 in \cite{cite:ruzsatriv}
by replacing invocation of Pl\"unnecke's inequality with
an application of theorem~\ref{gentriagineq}.

The rest of the paper is split into three sections. 
In the first section we gather the tools that we need 
from combinatorial number theory. The bulk 
of the paper is in the second section, that contains
the proofs of theorems~\ref{thmAthreeA} and~\ref{mainthm}
about the lower bounds on sums of dilates.
The final section contains the proof of the Pl\"unnecke-type 
theorem~\ref{gentriagineq}.

\section{Notation and tools}
\begin{lemma}[Sum form of Ruzsa's triangle inequality \cite{cite:ruzsa_plunnecke}, Corollary~6.2]
\label{ruzsatriag}
For any finite $A,B,C\subset \Z$ we have
\begin{equation*}
\abs{A+C}\leq\frac{\abs{A+B}\abs{B+C}}{\abs{B}}.
\end{equation*}

\end{lemma}
\begin{corollary}\label{triangcor}
For any finite sets $A,B\subset\Z$
\begin{equation*}
\abs{A+B}\geq \sqrt{\abs{A+A}\abs{B}}.
\end{equation*}
\end{corollary}
In proving theorem~\ref{gentriagineq}, 
besides the operation of forming a dilate $\lambda\cdot A$
we will also make use of the operation of repeated addition
$\lambda*A=\{a_1+\dotsb+a_{\lambda} : a_1,\dotsc,a_{\lambda}\in A\}$.
Moreover, we will need to be able to bound the size of sums 
of the form $\lambda_1*A+\lambda_2*A$ from above.
\begin{lemma}[Pl\"unnecke's inequality]
If $\abs{A+A}\leq K\abs{A}$ or $\abs{A-A}\leq K\abs{A}$, then
$\abs{\lambda_1*A-\lambda_2*A}\leq K^{\lambda_1+\lambda_2}\abs{A}$
for all non-negative integers $\lambda_1,\lambda_2$.
\end{lemma}
\begin{lemma}[Ruzsa's covering lemma, \cite{cite:ruzsa_covering}]\label{coveringlemma}
For any non-empty set $A,B$ in abelian group $G$ one can cover
$B$ by $\frac{\abs{A+B}}{\abs{A}}$ translates of $A-A$.
\end{lemma}
\begin{definition}
Let $G_1,G_2$ be abelian groups.
We say that $A_1\subset G_1$ and $A_2\subset G_2$ are \emph{$r$-isomorphic}
if there is a bijection $\phi\colon A_1\to A_2$ satisfying
\begin{equation*}
a_1+\dotsb+a_r=b_1+\dotsb+b_r\iff \phi(a_1)+\dotsb+\phi(a_r)=\phi(b_1)+\dotsb+\phi(b_r)
\end{equation*}
for all $a_1,\dotsb,a_r,b_1,\dotsc,b_r\in A_1$. The map $\phi$
is called a Freiman isomorphism of order~$r$.
\end{definition}
We will need the following version of Freiman's theorem which
can be deduced by a similar argument to the standard
Freiman's theorem.
%, but in place of Freiman
%isomorphisms of order eight using Freiman isomorphisms
%of higher order.
\begin{theorem}[\cite{cite:bilu_freiman}, Theorem~1.2]\label{freimanthm}
Fix $r\in \N$. Suppose a non-empty set 
$A\subset \Z$ satisfies $\abs{A+A}\leq K\abs{A}$.
Then $A$ is $r$-isomorphic to a subset of $[1,t_1]\times\dotsb\times[1,t_d]\subset
\Z^d$ of density at least $\alpha>0$, where $d$ and $\alpha$
depend only on $K$ and $r$, but not on $A$.
\end{theorem}

For a vector $\blambda=(\lambda_1,\dotsc,\lambda_k)$ we set
$\blambda^i=(\lambda_1,\dotsc,\lambda_{i-1},\lambda_{i+1},\dotsc,\lambda_k)$
and let $S_{\blambda}(A)=\lambda_1\cdot A+\dotsb+\lambda_k\cdot A$
denote the corresponding sumset involving~$A$. 
The greatest common divisor of a set of integers $\{\lambda_1,\dotsc,\lambda_k\}$ is
$\gcd(\lambda_1,\dotsc,\lambda_k\}=\max \{d\in\N : d\mid \lambda_i\text{ for  }i=1,\dotsc,k\}$.
Integers $\{\lambda_1,\dotsc,\lambda_k\}$ are said to be coprime
if $\gcd(\lambda_1,\dotsc,\lambda_k)=1$. For $\blambda=(\lambda_1,\dotsc,\lambda_k)$
we abbreviate $\gcd(\lambda_1,\dotsc,\lambda_k)$ to~$\gcd(\blambda)$. The
notation $\norm{\blambda}_1$ stands for 
$\abs{\lambda_1}+\dotsb+\abs{\lambda_k}$.

\section{Lower bounds on sums of dilates}
We start off with a lower bound on $A+3\cdot A$.
\begin{proof}[Proof of theorem~\ref{thmAthreeA}]
Let $a_1<\dotsb<a_n$ be the elements of $A$ in increasing order.
Let $A_k=\{a_1,\dotsc,a_k\}$. We will analyze how the size
of $B_k=A_k+3\cdot A_k$ grows as $k$ grows. We will prove that
for every $k\geq 5$ either $\abs{B_k}-\abs{B_{k-1}}\geq 4$ holds or
both $\abs{B_k}-\abs{B_{k-1}}=3$ and $\abs{B_{k+1}}-\abs{B_k}\geq 5$ hold.
The theorem will then follow.

Note that three sums $3a_{k-1}+a_k,a_{k-1}+3a_k,a_k+3a_k$ belonging
to $B_k$ are greater than any element of $B_{k-1}$, and thus do not
belong to $B_{k-1}$. Moreover the three sums are distinct since
$3a_{k-1}+a_k<a_{k-1}+3a_k<a_k+3a_k$. Therefore, to complete the proof
we need to analyze the case when $B_k\setminus B_{k-1}$ consists
of precisely these three sums. 

There are two cases: either $3a_k+a_{k-2}$ is in $B_{k-1}$ or it is not.
\renewcommand{\labelenumi}{Case \Alph{enumi}:}
\begin{enumerate}
\item If $3a_k+a_{k-2}\in B_{k-1}$, then since $3a_k+a_{k-2}>3a_{k-1}+a_{k-2}$,
it follows that $3a_k+a_{k-2}=4a_{k-1}$. By dilating and translating the set $A$
as necessary we can assume that $a_{k-2}=0$ and $a_{k-1}=3$. Then it follows
that $a_k=4$. Since $a_k+3a_{k-2}=4$ is larger than 
$a_{k-1}+3a_{k-2}=3+3\cdot 0$, it follows that
$a_k+3a_{k-2}=3a_{k-1}+a_t$ or $a_k+3a_{k-2}=3a_k+a_t$ for some $t<k$. 
Thus $a_t=-5$ or $a_t=-8$ for some $t\leq k-3$. 
In either case, since $a_{k-3}\geq -8$,
it follows that $a_{k-3}+3a_k\geq 4>3=a_{k-1}+3a_{k-2}$. Thus 
$a_{k-3}+3a_k=a_{k-2}+3a_{k-1}$, and $a_{k-3}=-3$. 
Therefore the five largest elements of $B_k$ are $16,15,13,12,9$.

Since $a_{k-3}\neq a_t\geq -8$, we also have $a_{k-4}\geq -8$.
Hence $a_{k-4}+3a_k\geq 4=a_k+3a_{k-2}$, implying that
either $a_{k-4}+3a_k=a_k+3a_{k-2}$ or $a_{k-4}+3a_k=a_{k-3}+3a_{k-1}$.
Therefore, either $a_{k-4}=-6$ or $a_{k-4}=-8$. Since $a_t$
does not exceed $a_{k-4}$, it follows that $a_t=-8$, after
all. 

As before $\{3\cdot 4+a_{k+1},4+3a_{k+1},a_{k+1}+3a_{k+1}\}\subset B_{k+1}\setminus B_k$. We need to exhibit two additional elements
from $B_{k+1}$ that are not in $B_k$. There are a few subcases to consider:
\begin{enumerate}
\item $3\cdot 4+a_{k+1}=3a_{k+1}+3$. In this case 
$a_{k+1}=9/2$ implying that $3a_{k+1}+0=27/2$
and $3a_{k+1}-3=21/2$ are not in $B_k$.
\item $3\cdot 4+a_{k+1}=3a_{k+1}+0$. In this case $a_{k+1}=6$
implying that $3a_{k+1}+a_{k-1}=21$ and 
$3a_{k+1}+a_t=10$ are not in $B_k$.
\item  $3\cdot 4+a_{k+1}=3a_{k+1}-3$. In this case $a_{k+1}=15/2$
implying that $3a_{k+1}+3=51/2$ and $3a_{k+1}+0=45/2$ are not
in $B_k$.
\end{enumerate}
If these subcases do not hold, then 
none of the four sums 
$3a_{k+1}+3$, 
$3a_{k+1}+0$, 
$3a_{k+1}-3$, 
$a_{k+1}+3\cdot 3$ 
are equal to any of the
three elements $\{3\cdot 4+a_{k+1},4+3a_{k+1},a_{k+1}+3a_{k+1}\}$ 
of $B_{k+1}\setminus B_k$ that we already counted. Moreover,
all of them are greater than $9$.  We will show that
at least two of these four sums are not in $B_k$.
 Since $3a_{k+1}+3>3a_k+3=15$, only two subcases remain:
\begin{enumerate}
\setcounter{enumii}{3} % To continue numbering of subcases
\item $3a_{k+1}+3=16$. In this case $a_{k+1}=13/3$ implying that $3a_{k+1}-3=10$ and $a_{k+1}+3\cdot 3=40/3$ are not in $B_k$.
\item $3a_{k+1}+3$ is not in $B_k$. Since $3a_{k+1}>3a_k=12$, then either $3a_{k+1}+0$ is not in $B_k$ or it is equal to one of $16,15$ or $13$. 
In the latter case $a_{k+1}+3\cdot 3$ is not in $B_k$ being equal to
$43/3$, $14$ and $40/3$ in these three cases respectively.
\end{enumerate}
\item If $3a_k+a_{k-2}\not\in B_{k-1}$, then $3a_k+a_{k-2}=3a_{k-1}+a_k$.
By dilating and scaling we can ensure that $a_{k-2},a_{k-1},a_k$ are
$0,2,3$ respectively. Since $a_k+3\cdot a_{k-2}=3$ is an element of $B_{k-1}$,
we necessarily have that $a_{k-3}\geq 3-3a_{k-1}=-3$.
Indeed, suppose on the other hand that 
$a_{k-3} + 3 a_{k-1} < 3 = a_k + 3 a_{k-2}$. Then we must have that 
$a_k + 3 a_{k-2}$ is equal to $a_{k-2} + 3 a_{k-1}$ or $4 a_{k-1}$, neither of which is possible.
Since $3a_k+a_{k-3}$
is an element of $B_{k-1}$ not less than $6$ there are two cases
\begin{enumerate}
\item $a_{k-3}=-3$. The five largest elements of $B_k$ are $12,11,9,8,6$.  The
sums $3a_{k+1}+a_{k+1}$, $3a_{k+1}+3$, $a_{k+1}+3\cdot 3$ are in $B_{k+1}\setminus B_k$. The sums $3a_{k+1}+2$, $3a_{k+1}+0$, $3a_{k+1}-3$, $a_{k+1}+6$ are each greater than $6$. The several subcases are
\begin{enumerate}
\item $3a_{k+1}+2=a_{k+1}+9$. Then $3a_{k+1}+0=21/2$ and $3a_{k+1}-3=15/2$
are in $B_{k+1}\setminus B_k$.
\item $3a_{k+1}+2=12$. Then $3a_{k+1}+0=10$ and $3a_{k+1}-3=7$ are in $B_{k+1}\setminus B_k$.
\item $3a_{k+1}+2$ is not in $B_k$. Then $3a_{k+1}$ is either $11$ or $12$. In either case $a_{k+1}+6$ is in $B_{k+1}\setminus B_k$.
\end{enumerate}
\item $a_{k-3}=-1$. The four largest elements of $B_k$ are $12,11,9,8$. The
sums $3a_{k+1}+a_{k+1}$, $3a_{k+1}+3$, $a_{k+1}+3\cdot 3$ are in $B_{k+1}\setminus B_k$. The sums $3a_{k+1}+2$, $3a_{k+1}+0$, $3a_{k+1}-1$, $a_{k+1}+6$ are each greater than $8$. The subcases are
\begin{enumerate}
\item $3a_{k+1}+2=a_{k+1}+3\cdot 3$. Then $3a_{k+1}+0=21/2$ and $3a_{k+1}-1=19/2$
are in $B_{k+1}\setminus B_k$.
\item $3a_{k+1}+2=12$. Then $3a_{k+1}=10$ and $a_{k+1}+6=28/3$ are in $B_{k+1}\setminus B_k$.
\item In the case $3a_{k+1}+2\in B_{k+1}\setminus B_k$ the sum $3a_{k+1}+0$ is
either $12$ or $11$. In the first case $a_{k+1}+6=10$ is in $B_{k+1}\setminus B_k$.
In the second case $3a_{k+1}-1=10$ is in $B_{k+1}\setminus B_k$ and not
equal to $a_{k+1}+9$.
\end{enumerate}
\end{enumerate}
\end{enumerate}
\end{proof}

One might be puzzled by the 
two-step induction scheme in the proof above, where
with addition of each next element
the sumset $B_k$ either grows
by the required number of elements, or
it grows by more than that at the next step.
However, this actually occurs for the set
$A=\{0,1,3,4,\dotsc,3k,3k+1\}$. Each next multiple of 
$3$ increases the size of the sumset by $5$, and 
every other number increases the sumset only by $3$.
Such examples impose a limitation on how simple such kinds 
of proofs can be. For instance, if one adopts this proof strategy 
to show that $\abs{A+4\cdot A}\geq 5\abs{A}-O(1)$,
then the example $A=\{0,1,2,4,5,6\dotsc,4m,4m+1,4m+2\}$
shows that a similar proof will require a three-step
induction. 

For proving the lower bound on an arbitrary sum of dilates 
\eqref{genericdilatesum} we abandon the proof strategy above.
The basis for the modified approach is the observation
that the reason why $A+\lambda\cdot A$ is large
in the examples above is that
$A$ can be partitioned into $\lambda$ subsets $A_1,\dotsc,A_{\lambda}$
according to the residue class modulo $\lambda$, such that
$A_i+\lambda\cdot A_i$ are disjoint from one another
for different values of~$i$.
Thus $A+\lambda\cdot A$ is large because each of 
$A_i+\lambda\cdot A_i$ is large.

For a general sum of dilates $S_{\blambda}(A)=\lambda_1\cdot A+\dotsb+\lambda_k\cdot A$ 
it turns out that looking modulo only $\lambda_1$
is insufficient. One needs to find a $\tau$ that is 
coprime with $\sum_{i=1}^k \lambda_i$. Then if
$A_1\cup \dotsb\cup A_{\tau}$ is a partition of $A$ into residue
classes modulo $\tau$, the $S_{\blambda}(A_i)$'s are disjoint.

It would have been excellent 
had $A_1,\dotsc,A_{\tau}$ always turned out to be arithmetic
progressions. They need not be, but under favorable circumstances 
at least one of the $A_i$ is a somewhat denser set than $A$.
The denser a set is, the closer it is to being an arithmetic
progression. So we would like to keep the dense sets of the partition.
As for the parts that are not dense, those will be partitioned
further into more parts, at least some of which are
dense. This leads to a recursive
subpartition process, where at each step we partition
sparse sets until only a few elements of $A$ belong to 
the sparse parts. Those we will discard.

The next lemma characterizes sets $A$ that cannot 
be broken into parts, at least one of which is dense, as those for which
one can use induction on the number
of dilates. After that 
lemma~\ref{partitionlemma} describes the basic step in the 
repeated subpartition process.
\begin{lemma}\label{inductlemma}
For a vector $\blambda=(\lambda_1,\dotsc,\lambda_k)$ of $k\geq 2$ coprime non-zero integers,
let $\tau_i=\gcd(\blambda^i)$.
%\begin{equation*}
%\tau_i=\gcd(\lambda_1,\dotsc,\lambda_{i-1},\lambda_{i+1},\dotsc,\lambda_k).
%\end{equation*}
Then for every such $\blambda$, every $\delta>0$ and every
finite set $A\subset \{1,\dotsc,n\}$ at least one of the following holds:
\renewcommand{\labelenumi}{\Roman{enumi})}
\renewcommand{\theenumi}{\Roman{enumi}}
\begin{enumerate}
\item \label{inductalt} The sumset 
%$S=\lambda_1\cdot A+\dotsb+\lambda_k\cdot A$ 
$S=S_{\blambda}(A)$
satisfies
\begin{equation*}
\abs{S}\geq \frac{1}{k-1}\sum_{i=1}^k \tau_i\abs{S_i}
-2\delta n-\tau_1
\end{equation*}
where $S_i=S_{\blambda^i}(A)$.
%\begin{equation*}
%S_i=\lambda_1\cdot A+\dotsb+\lambda_{i-1}\cdot A+\lambda_{i+1}\cdot A+\dotsb+\lambda_k\cdot A.
%\end{equation*}
\item \label{sparsealt} There is an $i\in\{1,\dotsc,k\}$ and $r^*\in \Z/\tau_i\Z$ such that
the set
\begin{equation*}
\{a\in A: a\equiv r^* \pmod {\tau_i}\}
\end{equation*}
is contained either in $\bigl[1,\bigl(1-\delta/\abs{\lambda_i}\bigr)n\bigr]$ or in $\bigl[\delta n/\abs{\lambda_i},n\bigr]$.
\end{enumerate}
\end{lemma}
\begin{proof}
Suppose the alternative \ref{sparsealt} does not hold.
Then
\begin{align*}
l_{i,j}&=\min\{\lambda_i a : a\in A,\ a\equiv j\pmod{\tau_i}\},\\
r_{i,j}&=\max\{\lambda_i a : a\in A,\ a\equiv j\pmod{\tau_i}\},\\
l_i&=\min\{\lambda_i A\}=\min_j l_{i,j},\\
r_i&=\max\{\lambda_i A\}=\max_j r_{i,j}
\end{align*}
satisfy $l_{i,j}-l_i\leq \delta n$ and $r_i-r_{i,j}\leq \delta n$.
Set $L_i=\{l_{i,1},\dotsc,l_{i,\tau_i}\}$ and
$R_i=\{r_{i,1},\dotsc,r_{i,\tau_i}\}$. Since
all elements of $S_i$ are divisible by $\tau_i$, whereas
$L_i$ is a set of distinct integers
modulo $\tau_i$, we have that $S_i+l_{i,j}$ is disjoint from $S_i+l_{i,j'}$
for $j\neq j'$. Similarly, $S_i+r_{i,j}$ is disjoint from $S_i+r_{i,j'}$.

For a set $T$ and $x\in \Z$ let $T_{\leq x}=\{t\in T: t\leq x\}$ and
$T_{>x}=\{t\in T: t>x\}$.
Now we will use the idea from the proof of \cite[theorem 1.1]{cite:supadd_submul}. Namely, we make
$k-1$ copies of the set $S$, and then mark some of the elements in each copy.
We allow some elements to be marked more than once.
We start by marking in the first copy the elements of $L_k+S_k$.
They all belong to the interval $[l_1+\dotsb+l_k,r_1+\dotsb+r_{k-1}+l_k+\delta n]$. 
Then in the first copy mark the elements of 
$R_{k-1}+(S_{k-1})_{>r_1+\dotsb+r_{k-2}+l_k}$.
All elements in the first copy are marked at most once except 
possibly some of the elements in the interval $[r_1+\dotsb+r_{k-1}+l_k-\delta n,r_1+\dotsb+r_{k-1}+l_k+\delta n]$
are marked twice.
This interval has length $2\delta n$.

Then, for $2\leq i\leq k-2$, in the $i$'th copy we mark the elements of 
\begin{equation*}
L_{k-i+1}+(S_{k-i+1})_{\leq r_1+\dotsb+r_{k-i}+l_{k-i+2}+\dotsb+l_k}
\end{equation*}
and of
\begin{equation*}
R_{k-i}+(S_{k-i})_{>r_1+\dotsb+r_{k-i-1}+l_{k-i+1}+\dotsb+l_k}.
\end{equation*} 
Only the elements in the interval 
\begin{equation*}
[r_1+\dotsb+r_{k-i}+l_{k-i+1}+\dotsb+l_k-\delta n,r_1+\dotsb+r_{k-i}+l_{k-i+1}+\dotsb+l_k+\delta n]
\end{equation*}
can possibly be marked twice.
Finally, in the $k-1$'st copy we mark the elements of
$L_2+(S_2)_{\leq r_1+l_3+\dotsb+l_k}$ and of
$R_1+(S_1)_{>l_2+\dotsb+l_k}$. Again elements only in 
$[r_1+l_2+\dotsb+l_k-\delta n,r_1+l_2+\dotsb+l_k+\delta n]$
can be marked twice. And again this interval
is of length $2\delta n$.

Counting the number of marked elements we obtain
\begin{equation*}
(k-1)\abs{S}\geq \sum_{i=1}^k \tau_i\abs{S_i}-2(k-1)\delta n-\tau_1
\end{equation*}
where the right side counts the number of
elements that are marked at least once, and
the left side counts the total number of elements. The term $\tau_1$
on the right is due to the set $(S_1)_{>l_2+\dotsb+l_k}$ having one fewer element
than $S_1$.
\end{proof}
\begin{corollary}\label{fibercorollary}
If $\sum_{i=1}^k \lambda_i=0$, then 
\begin{equation}\label{fibercoreq}
\abs{S_{\blambda}(A)}\geq \frac{1}{k-1}\sum_{i=1}^k\abs{S_{\blambda}(A^i)}-5
\end{equation}
for any non-empty $A\subset \Z$. Moreover the vectors $\blambda^i$ are
coprime for every~$i$.
\end{corollary}
\begin{proof}
Since both sides of \eqref{fibercoreq} are translation-invariant,
we can assume that $1\in A$, and set $n=\max A$.
Since $\tau_i\mid \sum_{j\neq i}\lambda_j=-\lambda_i$ and
$\blambda$ is a coprime vector, $\tau_i=1$ for all $i$. If we set
$\delta=2/n$, then the alternative \ref{sparsealt} does not hold,
and the alternative \ref{inductalt} becomes \eqref{fibercoreq}.
\end{proof}

\begin{lemma}\label{partitionlemma}
For every $\blambda=(\lambda_1,\dotsc,\lambda_k)$ satisfying
$\sum_{i=1}^k \lambda_i\neq 0$ there
are $\alpha>0$ and $\beta>0$ such that
for every integer $t\geq 0$ and every  
finite set $A\subset \Z$ of size $\abs{A}\geq (\alpha/2)^{-\beta^t}$
there are four families of sets $\mathcal{D}_t$, $\mathcal{G}_t$,
$\mathcal{S}_t, \mathcal{T}_t$ satisfying
\begin{enumerate}
\item The families $\mathcal{D}_t$, $\mathcal{G}_t$,
$\mathcal{S}_t, \mathcal{T}_t$ together form a partition of $A$, i.e., the
sets in $\mathcal{D}_t$, $\mathcal{G}_t$,
$\mathcal{S}_t, \mathcal{T}_t$ are disjoint from one another and their union is $A$.
\item If $B_1,B_2$ are any two unequal sets from $\mathcal{D}_t\cup
\mathcal{G}_t\cup\mathcal{S}_t\cup \mathcal{T}_t$ (i.e. the sets $B_1$ and
$B_2$ possibly belong to different families), then $S_{\blambda}(B_1)$
is disjoint from $S_{\blambda}(B_2)$.
\item The sets in $\mathcal{D}_t$ are dense: each set in $\mathcal{D}_t$
is $\norm{\blambda}_1$-isomorphic to a subset of an interval
of length at least $\abs{A}^{\beta^t}$ of density
at least $\alpha/2$.
\item The sets in $\mathcal{G}_t$ are growing: for each $G\in\mathcal{G}_t$
we have $\abs{S_{\blambda}(G)}\geq \norm{\blambda}_1\abs{G}$.
\item The sets in $\mathcal{S}_t$ are small, but not too small: $\bigl\lvert\bigcup \mathcal{S}_t\bigr\rvert \leq \abs{A}/2^t$,
but $\abs{S}\geq \abs{A}^{\beta^t}$ for every $S\in\mathcal{S}_t$.
\item The sets in $\mathcal{T}_t$ are tiny: $\bigl\lvert\bigcup \mathcal{T}_t\bigr\rvert \leq \frac{1}{\alpha}\sum_{i=1}^t \abs{A}^{1-\beta^i}$.
\end{enumerate}
\end{lemma}
\begin{proof}
We let $\alpha$ and $d$ be as in Freiman's theorem (theorem~\ref{freimanthm}) when
applied with $K=\norm{\blambda}_1^2$ and~$r=\norm{\blambda}_1$. We set $\beta=1/2d$.

The proof is by induction on $t$. For $t=0$ we simply set
$\mathcal{D}_0=\mathcal{G}_0=\mathcal{T}_0=\emptyset$ and $\mathcal{S}_0=\{A\}$.
If $t\geq 1$, then we use induction to obtain
$\mathcal{D}_{t-1}$, $\mathcal{G}_{t-1}$, $\mathcal{S}_{t-1}$ and $\mathcal{T}_{t-1}$.
We will not do anything to sets in $\mathcal{D}_{t-1}$, $\mathcal{G}_{t-1}$ and $\mathcal{T}_{t-1}$, 
they will become members of $\mathcal{D}_{t}$, $\mathcal{G}_{t}$ and $\mathcal{T}_t$
respectively.
However, the sets in $\mathcal{S}_{t-1}$ will be either moved to
$\mathcal{G}_{t}$ or subpartitioned further into $\mathcal{D}$-, $\mathcal{S}$-
and $\mathcal{T}$-sets.

Let $A'$ be a set in $\mathcal{S}_{t-1}$.
If $\abs{\lambda_1\cdot A'+\lambda_1\cdot A'}\geq \norm{\blambda}_1^2\abs{A'}$, then by corollary~\ref{triangcor} 
\begin{align*}
\abs{\lambda_1\cdot A'+(\lambda_2\cdot A'+\dotsb+\lambda_k\cdot A')}&\geq 
\norm{\blambda}_1\sqrt{\abs{A'}\abs{\lambda_2\cdot A'+\dotsb+\lambda_k\cdot A'}}
\\&\geq \norm{\blambda}_1\abs{A'}
\end{align*}
and we can move $A'$ to $\mathcal{G}_t$.

Hence we can assume that $\abs{A'+A'}=\abs{\lambda_1\cdot A'+\lambda_1\cdot A'}<
\norm{\blambda}_1^2\abs{A'}$. By Freiman's theorem (theorem~\ref{freimanthm})
the set $A'$ is $\norm{\blambda}_1$-isomorphic to $A''$ which is a subset of
$[1,t_1]\times\dotsb\times[1,t_d]$ of density at least $\alpha>0$,
where $d$ and $\alpha$ as above. Since
$A'$ and $A''$ are $\norm{\blambda}_1$-isomorphic,
$\abs{S_{\blambda}(A')}=\abs{S_{\blambda}(A'')}$.
Without loss of generality we may assume that $t_1\geq \dotsb\geq t_d$.
This assures us that $t_1\geq \abs{A'}^{1/d}\geq \abs{A}^{\beta^{t-1}/d}$.
For every $x\in [1,t_2]\times\dotsb\times[1,t_d]$ there is
a \textit{fiber} 
\begin{equation*}
A_x=\{(a_1,\dotsc,a_d)\in A'' : (a_2,\dotsc,a_d)=x\}.
\end{equation*}
These fibers form a partition
of~$A''$. 
Since $\sum_{i=1}^k \lambda_i\neq 0$ the set
$S_{\blambda}(A_x)$ is disjoint from $S_{\blambda}(A_y)$ for $x\neq y$.

Let $X=\{x\in [1,t_2]\times\dotsb\times[1,t_d] : \abs{A_x}\leq \alpha t_1/2\}$.
For $x\not\in X$ the fiber $A_x$ is
$\norm{\blambda}_1$-isomorphic to a subset of the interval $[1,t_1]$ of density at least
$\alpha/2$. Since $t_1\geq \abs{A}^{2\beta^t}\geq 
\frac{2}{\alpha}\abs{A}^{\beta^t}$ we can move
any fiber $A_x$ with $x\not\in X$ to $\mathcal{D}_t$.

Let $Y=\{y\in [1,t_2]\times\dotsb\times[1,t_d] : \abs{A_y}\leq t_1^{1/2}\}$.
Since $\abs{A'}\geq \alpha t_1\dotsb t_d$  we have that
 $\abs{\bigcup_{y\in Y} A_y}\leq (\abs{A'}/\alpha)t_1^{-1/2}$.
Therefore the total number of elements in fibers
of the form $A_y$ for $y\in Y$ for all $A'\in S$ is
at most $(\abs{A}/\alpha)\abs{A}^{-\beta^t}$. 
We add $\{A_y\}_{y\in Y}$ to $\mathcal{T}_t$.  
Because $\abs{\bigcup_{x\in X\setminus Y} A_x}\leq \abs{\bigcup_{x\in X} A_x} \leq \abs{A'}/2$, the remaining fibers $A_x$ with $x\in X\setminus Y$
can then be moved to~$\mathcal{S}_t$.
\end{proof}

With the previous two lemmas in our arsenal, we are ready to prove
the sharp lower bound on the arbitrary sum of dilates.  
\begin{proof}[Proof of theorem~\ref{mainthm}]
The proof is by induction on~$k$. The case $k=1$ is true since $\gcd(\lambda_1)=1$
only if $\lambda_1\in\{\pm 1\}$. Suppose we are given a vector 
$\blambda=(\lambda_1,\dotsc,\lambda_k)$ of coprime integers. 
Assume we have already established the theorem for all vectors of fewer than $k$ integers.
We can assume that $\sum \lambda_i\neq 0$ since 
in the case $\sum \lambda_i=0$
corollary~\ref{fibercorollary} yields
\begin{align*}
\abs{S_{\blambda}(A)}&\geq \frac{1}{k-1}\sum_{i=1}^k 
S_{\blambda^i}\abs{A}-5\\
&\geq \frac{1}{k-1}\sum_{i=1}^k 
\bigl(\norm{\blambda^i}_1\abs{A}-o(\abs{A})\bigr)-5\\
&=\frac{1}{k-1}\sum_{i=1}^k \bigl(\norm{\blambda}_1-\abs{\lambda_i}\bigr)\abs{A}-o(\abs{A})\\
&=\norm{\blambda}_1\abs{A}-o(\abs{A})
\end{align*}
by the induction hypothesis.

Let $M$ be the largest number such that $\abs{S_{\blambda}(A)}\geq
M\abs{A}-o(\abs{A})$. Similarly, let $M(\gamma)$ be the largest
number such that $\abs{S_{\blambda}(A)}\geq
M(\gamma)\abs{A}-o(\abs{A})$ for sets $A$ that are
subsets of density at least~$\gamma$ in some interval.

\begin{claim}\label{firstclaim} $M\geq M(\alpha/2)$ where $\alpha$ is as in
lemma~\ref{partitionlemma}.
\end{claim}
\begin{claim}\label{secondclaim}For every $\delta>0$ and $0<\gamma<1$
\begin{equation*}
M(\gamma)\geq \min\left(M+\Bigl(M\bigl(\gamma(1+\delta/4\norm{\blambda}_{\infty}^2)\bigr)-M\Bigr)\frac{\delta}{4\norm{\blambda}_{\infty}^3},\norm{\blambda}_1-2\delta/\gamma\right).
\end{equation*}
In case $\gamma(1+\delta/4\norm{\blambda}_{\infty}^2)>1$ the right hand side
should be interpreted as~$\norm{\blambda}_1-2\delta/\gamma$.
\end{claim}

\begin{proof}[Proof of claim~\ref{firstclaim}]
Let $\alpha$ and $\beta$ be as in lemma~\ref{partitionlemma}.
Fix $\epsilon>0$ and an integer $t\geq 0$.
Let $N$ be so large that $\abs{S_{\blambda}(A)}\geq
\bigl(M(\alpha/2)-\epsilon\bigr)\abs{A}$ for sets $A$
with at least $N$ elements that are
subsets of density at least~$\alpha/2$ in some interval.
lemma~\ref{partitionlemma} then implies that
$\abs{S_{\blambda}(A)}\geq \bigl(M(\alpha/2)-\epsilon\bigr)(1-2^{-t}-\frac{1}{\alpha}\sum_{i=1}^t \abs{A}^{-\beta^i})\abs{A}$
if $\abs{A}\geq N^{\beta^{-t}}$. Therefore,
$M\geq \bigl(M(\alpha/2)-\epsilon\bigr)(1-2^{-t})$ for
every $\epsilon>0$ and every $t\geq 0$.
\end{proof}
\begin{proof}[Proof of claim~\ref{secondclaim}]
Fix an $\epsilon>0$. Let $N$ be so large that
$\abs{S_{\blambda}(A)}\geq \bigl(M(\epsilon)-\epsilon\bigr)\abs{A}$ for sets $A$
with at least $N$ elements that are subsets of intervals of density at least~$\epsilon$.
Moreover, let also $N$ be so large that
$\abs{S_{\blambda}(A)}\geq \Bigl(M\bigl(\gamma(1+\delta/4\abs{\lambda_i}\tau_i)\bigr)-\epsilon\Bigr)\abs{A}$ 
for sets $A$ that are subsets of density at least~$\gamma(1+\delta/4\abs{\lambda_i}\tau_i)$ in some interval.

Let $A$ be a subset of an interval of density at least $\gamma$,
and suppose $\abs{A}\geq N\max_i \tau_i/\epsilon$.
Without loss of generality we may assume that
$A\subset\{1,\dotsc,n\}$ and $1,n\in A$.
Apply lemma~\ref{inductlemma} with $\delta$ as in
the statement of the claim. If 
the alternative~\ref{inductalt} holds, then
\begin{align*}
S_{\blambda}(A)&\geq \frac{1}{k-1}\sum_{i=1}^k \tau_i\abs{S_{\blambda^i}(A)}-2\delta n-\tau_1\\
&\geq \frac{1}{k-1}\sum_{i=1}^k \tau_i\abs{A}\norm{\blambda^i/\tau_i}_1-o(\abs{A})-2\delta n-\tau_1\\
&=\frac{1}{k-1}\sum_{i=1}^k \bigl(\norm{\blambda}-\abs{\lambda_i}\bigr)-2\delta n-o(n)\\
&=\norm{\blambda}_1\abs{A}-2\delta n-o(n)\\
&\geq \bigl(\norm{\blambda}_1-2\delta/\gamma-o(1)\bigr)\abs{A}.
\end{align*}
Suppose the alternative~\ref{sparsealt} holds, and let $i$ and $r^*$ be
given as in the alternative. For $r\in \Z/\tau_i\Z$ 
let $A_r=\{a\in A: a\equiv r\pmod{\tau_i}\}$. Since $\blambda$ is
a coprime vector, $\tau_i$ is coprime with $\sum_{j=1}^k \lambda_j$.
Therefore $S_{\blambda}(A_{r_1})$ is disjoint from
$S_{\blambda}(A_{r_2})$ for~$r_1\neq r_2$. Let 
$B_r=\{\lfloor a/\tau_i\rfloor : a\in A_r\}$. Clearly
$\abs{S_{\blambda}(A_r)}=\abs{S_{\blambda}(B_r)}$. Each set
$B_r$ is contained in an interval of length 
$\lceil n/\tau_i\rceil$. Moreover, $B_{r^*}$ is contained
in a shorter interval of length $\lceil n(1-\delta/\abs{\lambda_i})/\tau_i\rceil$.
Therefore the total length of the intervals containing $\{B_r\}_{r\in\Z/\tau_i\Z}$
is at most
\begin{equation*}
n\left(1-\frac{\delta}{\abs{\lambda_i}\tau_i}\right)+\tau_i
\leq n\left(1-\frac{\delta}{2\abs{\lambda_i}\tau_i}\right).
\end{equation*}
Let $R_1=\{r\in\Z/\tau_i\Z : \abs{B_r}\leq \abs{A}\delta/4\abs{\lambda_i}\tau_i^2\}$.
Then $\sum_{r\in R_1}\abs{B_r}\leq \abs{A}\delta/4\abs{\lambda_i}\tau_i$.
Therefore there is an $r_0\in (\Z/\tau_i\Z)\setminus R_1$ such that
the density of $B_{r_0}$ in the appropriate interval is
at least 
\begin{equation*}
\frac{\abs{A}-\sum_{r\in R_1}\abs{B_r}}
{n(1-\delta/2\abs{\lambda_i}\tau_i)}\geq 
\gamma\frac{1-\delta/4\abs{\lambda_i}\tau_i}{1-\delta/2\abs{\lambda_i}\tau_i}
\geq \gamma(1+\delta/4\abs{\lambda_i}\tau_i).
\end{equation*}
Let $R_2=\{r\in\Z/\tau_i\Z : \abs{B_r}\leq \epsilon\abs{A}/\tau_i\}$. Then
\begin{align*}
S_{\blambda}(A)&\geq \sum_{r\in \Z/\tau_i\Z}\abs{S_{\blambda}(B_r)}\\
&\geq \sum_{r\in (\Z/\tau_i\Z)\setminus R_2}\abs{S_{\blambda}(B_r)}\\
&=\abs{S_{\blambda}(B_{r_0})}+\sum_{r\in (\Z/\tau_i\Z)\setminus (R_2\cup\{r_0\})}\abs{S_{\blambda}(B_r)}\\
&\geq \Bigl(M\bigl(\gamma(1+\delta/4\abs{\lambda_i}\tau_i)\bigr)-\epsilon\Bigr)\abs{B_{r_0}}
+\sum_{r\in (\Z/\tau_i\Z)\setminus (R_2\cup\{r_0\})}(M(\epsilon)-\epsilon)\abs{B_r}\\
&\geq (M(\epsilon)-\epsilon)\abs{A}(1-\epsilon)+
\Bigl(M\bigl(\gamma(1+\delta/4\abs{\lambda_i}\tau_i)\bigr)-M(\epsilon)\Bigr)\abs{B_{r_0}}.
\end{align*}
Since $M(\epsilon)\geq M$ and $\epsilon$ can be chosen arbitrarily small, we infer
\begin{equation*}
M(\gamma)\geq \min\left(M+\Bigl(M\bigl(\gamma(1+\delta/4\abs{\lambda_i}\tau_i)\bigr)-M\Bigr)\frac{\delta}{4\abs{\lambda_i}\tau_i^2},\norm{\blambda}_1-2\delta/\gamma\right).
\end{equation*}
Since $\tau_i^2\abs{\lambda_i}\leq \norm{\blambda}_{\infty}^3$ the 
claim~\ref{secondclaim} follows.
\end{proof}
Claims \ref{firstclaim} and \ref{secondclaim} imply the theorem. Indeed,
fix $\delta>0$ and assume that 
$M\leq \norm{\blambda}_1-4\delta/\alpha$.
Let 
\begin{equation*}
\Gamma=\bigl\{\gamma\in [\alpha/2,1] : M\geq M(\gamma)\bigr\}.
\end{equation*}
By claim~\ref{firstclaim} the set $\Gamma$ is non-empty.
By claim~\ref{secondclaim} $\gamma\in\Gamma$ implies that either
\begin{equation*}
M\geq M(\gamma)\geq \norm{\blambda}_1-2\delta/\gamma
\end{equation*}
which is inconsistent with the assumption above, or that
\begin{equation*}
M\geq M(\gamma)\geq M+\Bigl(M\bigl(\gamma(1+\delta/4\norm{\blambda}_{\infty}^2)\bigr)-M\Bigr)\frac{\delta}{4\norm{\blambda}_{\infty}^3}
\end{equation*}
implying
\begin{equation*}
M\geq M\bigl(\gamma(1+\delta/4\norm{\blambda}_{\infty}^2)\bigr)
\end{equation*}
and $\gamma(1+\delta/4\norm{\blambda}_{\infty}^2)\in \Gamma$. However,
this is a contradiction since no element in $\Gamma$ exceeds~$1$.
Thus
$M\geq \norm{\blambda}_1-4\delta/\alpha$ for every $\delta>0$,
and it follows that $M\geq \norm{\blambda}_1$.
\end{proof}
%Back of envelope calculations give
%For n-element set A we have
%\abs{S_{\blambda}(A)}>=\norm{\blambda}_1 n - O(n log^2 log log n/log log n)

\section{Pl\"unnecke-type inequalities on sums of dilates}
\begin{proof}[Proof of theorem~\ref{gentriagineq}]
First we deal with the case $\lambda_1,\dotsc,\lambda_k>0$. Without loss of generality
$0\in A$. 
Let $r=\max_i \lfloor \log_2 \lambda_i\rfloor$. Write $\lambda_i$ in the base $2$ as 
%\begin{equation*}
$
\lambda_i=\sum_{j=0}^{r}
\lambda_{i,j} 2^j%.
$ with $\lambda_{i,j}\in\{0,1\}$.
%\end{equation*}
Then clearly
\begin{equation}\label{digitinclusion}
S_{\blambda}(A)\subset \sum_{j=0}^r \Bigl(\sum_{i=1}^k \lambda_{i,j}\Bigr)*(2^j\cdot A).
\end{equation}
Since by Pl\"unnecke's inequality $\abs{t*(2\cdot A)+2\cdot A-2\cdot A+A}\leq
\abs{(2t+3)*A-2*A}\leq K^{2t+5}\abs{A}$, lemma~\ref{coveringlemma}
implies that there are $X_1,\dotsc,X_r$ satisfying
\begin{equation}\label{coveringinclusion}
\Bigl(\sum_{i=1}^k \lambda_{i,j}\Bigr)*(2\cdot A)+2\cdot A-2\cdot A\subset
A-A+X_j, \qquad\abs{X_j}\leq K^{2\sum_{i=1}^k\lambda_{i,j}+5}.
\end{equation}
Let $Y_j=2^{j-1}\cdot X_j$ for $j=1,\dotsc,r$. Note that $\abs{Y_j}=\abs{X_j}$.
The inclusion \eqref{coveringinclusion} with $j=r$ and \eqref{digitinclusion} combine into
\begin{equation*}
S_{\blambda}(A)\subset \sum_{j=0}^{r-1} 
\Bigl(\sum_{i=1}^k \lambda_{i,j}\Bigr)*(2^j\cdot A)+2^{r-1}\cdot A-2^{r-1}\cdot A+Y_r.
\end{equation*}
Repeatedly using inclusion \eqref{coveringinclusion} for $j=r-1,r-2,\dotsc,1$ we obtain
\begin{equation*}
S_{\blambda}(A)\subset A-A+\Bigl(\sum_{i=1}^k \lambda_{i,0}\Bigr)*A+Y_1+\dotsb+Y_r
\end{equation*}
implying
\begin{align*}
\abs{S_{\blambda}(A)-A}&\leq \Bigl\lvert A-A-A+\Bigl(\sum_{i=1}^k \lambda_{i,0}\Bigr)*A\Bigr\rvert\prod_{j=1}^r
\abs{X_j}\\
&\leq \abs{A}K^{3+5r+2\sum_{j=0}^r\sum_{i=1}^k \lambda_{i,j}}.
\end{align*}

Now we turn to the case when some of $\lambda$'s are negative. Say
$\lambda_1,\dotsc,\lambda_p$ are positive, whereas $\lambda_{p+1},
\dotsc,\lambda_k$ are negative. As before we let $\abs{\lambda_i}=\sum_{j=0}^r 
\lambda_{i,j}2^j$. Let
$B=\lambda_1\cdot A+\dotsb+\lambda_p\cdot A$ to be the sum
of positive dilates, and $C=\lambda_{p+1}\cdot A+\dotsb+\lambda_k\cdot A$
to be the sum of negative dilates. 
By above 
\begin{align*}
\abs{B+A}&\leq \abs{A}K^{4+5r+2\sum_{j=0}^r\sum_{i=1}^p \lambda_{i,j}},\\
\abs{C+A}&\leq \abs{A}K^{3+5r+2\sum_{j=0}^r\sum_{i=p+1}^k \lambda_{i,j}}.\\
\end{align*}
By the triangle inequality (lemma~\ref{ruzsatriag})
\begin{equation*}
\abs{B+C}\leq \frac{\abs{A+B}\abs{A+C}}{\abs{A}}\leq 
\abs{A}K^{7+10r+2\sum_{j=0}^r\sum_{i=1}^k \lambda_{i,j}}.
\end{equation*}
Since $\abs{\sum_{j=0}^r \lambda_{i,j}}\leq \log_2 (1+\abs{\lambda_i})$
and $r\leq \max_i \log_2 (1+\abs{\lambda_i})$ the theorem follows.
\end{proof}
Observe that the actual bound obtained in the course of the proof of 
theorem~\ref{gentriagineq} involves the sum of binary digits of $\lambda_i$
rather than $\log(1+\abs{\lambda_i})$. In particular if 
$\lambda_1,\dotsc,\lambda_k$ are $k$ positive integers not exceeding $2^k$, each
containing no more than $7$ ones in binary development, then 
$\abs{A+A}\leq K\abs{A}$ implies $\abs{S_{\blambda}(A)}\leq K^{100 k}\abs{A}$.
Since the proof of theorem~\ref{gentriagineq} could be easily
adapted to use $b$-ary expansion in place of binary, similar results
are true of $\lambda$'s that have sparse $b$-ary expansion at the cost
of worsening the constant $100$ above if $b$ gets large.
Since numbers with $7$ ones in binary development 
are commonly believed to look quite random
in almost any other base, any bound that depends
on the base in which a number is written,
is unnatural. Perhaps, a condition
on the size of $\lambda$'s is all one needs:
\begin{question}
Suppose $\blambda=(\lambda_1,\dotsc,\lambda_k)$ satisfies 
$\abs{\lambda_i}\leq 2^k$, does it follow that
\begin{equation*}
\frac{\abs{\lambda_1\cdot A+\dotsb+\lambda_k\cdot A}}{\abs{A}}\leq 
\left(\frac{\abs{A+A}}{\abs{A}}\right)^{C k}
\end{equation*}
for an absolute constant~$C$?
\end{question}
\noindent With $C k^2$ in place of $C k$ the estimate follows from theorem~\ref{gentriagineq}.

One can also use the triangle inequality for proving inequalities similar to that in theorem~\ref{gentriagineq}:
\begin{theorem} For any $k\in\N$ and $\lambda\in\Z$ we have
\begin{equation*}
\frac{\abs{A+\lambda^k\cdot A}}{\abs{A}}\leq \left(\frac{\abs{A+A}}{\abs{A}}\right)^{k(\abs{\lambda}+1)}.
\end{equation*}
\end{theorem}
\begin{proof}
By the triangle inequality
\begin{align*}
\abs{A+(\lambda_1\lambda_2)\cdot A}&
\leq \frac{\abs{A+\lambda_1\cdot A}\abs{\lambda_1\cdot A+(\lambda_1 \lambda_2)\cdot A}}{\abs{\lambda_1\cdot A}}
=\frac{\abs{A+\lambda_1\cdot A}\abs{A+\lambda_2\cdot A}}{\abs{A}}
\end{align*}
and the theorem follows from Pl\"unnecke's inequality by induction on~$k$.
\end{proof}
Though a more careful argument can improve on the constants in theorem~\ref{gentriagineq}, the simplest case of $A+2\cdot A$ seems to be out of reach.  
\begin{question}
Is $\abs{A+2\cdot A}/\abs{A}\leq (\abs{A+A}/\abs{A})^p$ for
some $p<3$?
\end{question}\vspace{1ex}

\noindent \textbf{Acknowledgements.} I am grateful to Brooke Orosz, and
the anonymous referee who
read preliminary versions of the paper, and pointed out many 
inaccuracies.  This work was inspired by conversations
with Jacob Fox and Jacob Tsimerman.

After this paper was completed, I was informed by Manuel Silva
that he with Javier Cilleruelo and Carlos Vinuesa proved that
$\abs{A+3\cdot A}\geq 4\abs{A}-4$ and characterized the cases
when equality occurs.

\bibliographystyle{alpha}
\bibliography{dilates}

\newcommand{\etalchar}[1]{$^{#1}$}
\begin{thebibliography}{NOO{\etalchar{+}}07}

\bibitem[Bil99]{cite:bilu_freiman}
Yuri Bilu.
\newblock Structure of sets with small sumset.
\newblock {\em Ast\'erisque}, (258):77--108, 1999.
\newblock Structure theory of set addition.

\bibitem[Buk08]{cite:ruzsatriv}
Boris Bukh.
\newblock Non-trivial solutions to a linear equation in integers.
\newblock {\em Acta Arith.}, 131(1):51--55, 2008.
\newblock \hhref{math/0703767}.

\bibitem[Gar07]{cite:garaev_sumprod}
M.~Z. Garaev.
\newblock An explicit sum-product estimate in {$\mathbb{F}_p$}.
\newblock {\em Int. Math. Res. Not. IMRN}, (11):Art. ID rnm035, 11, 2007.
\newblock \hhref{math/0702780v1}.

\bibitem[GRM07]{cite:supadd_submul}
Katalin Gyarmati, Imre~Z. Ruzsa, and M{\'a}t{\'e} Matolcsi.
\newblock A superadditivity and submultiplicativity property for cardinalities
  of sumsets.
\newblock \hhref{0707.2707v1}, July 2007.

\bibitem[K{\L}06]{cite:konyagin_laba_transcendental}
Sergei Konyagin and Izabella {\L}aba.
\newblock Distance sets of well-distributed planar sets for polygonal norms.
\newblock {\em Israel J. Math.}, 152:157--179, 2006.
\newblock \url{http://www.math.ubc.ca/~ilaba/preprints/polyg_distances.pdf}.

\bibitem[KS07]{cite:katz_shen_sumprod}
Nets~Hawk Katz and Chun-Yen Shen.
\newblock A slight improvement to {G}araev's sum product estimate.
\newblock \hhref{math/0703614v1}, Mar 2007.

\bibitem[Nat07]{cite:nathanson_forms}
Melvyn~B. Nathanson.
\newblock Inverse problems for linear forms over finite sets of integers.
\newblock \hhref{0708.2304v2}, Aug 2007.

\bibitem[NOO{\etalchar{+}}07]{cite:binary_forms}
Melvyn~B. Nathanson, Kevin O'Bryant, Brooke Orosz, Imre Ruzsa, and Manuel
  Silva.
\newblock Binary linear forms over finite sets of integers.
\newblock {\em Acta Arith.}, 129:341--361, 2007.
\newblock \hhref{math/0701001}.

\bibitem[Ruz89]{cite:ruzsa_plunnecke}
Imre~Z. Ruzsa.
\newblock An application of graph theory to additive number theory.
\newblock {\em Scientia, Series A. Official journal of Universidad T{\'e}cnica
  Federico Santa Mar{\'\i}a}, 3:97--109, 1989.

\bibitem[Ruz99]{cite:ruzsa_covering}
Imre~Z. Ruzsa.
\newblock An analog of {F}reiman's theorem in groups.
\newblock {\em Ast\'erisque}, (258):323--326, 1999.
\newblock Structure theory of set addition.

\end{thebibliography}

\end{document}